\documentclass[12pt]{amsart}
\usepackage{amsfonts,amssymb,latexsym}
\usepackage{amsthm,cite,amsmath,amstext,amsbsy,bm,mathrsfs,url} 
\usepackage{paralist}
\usepackage{datetime} 
\usepackage{color}
\usepackage[dvipsnames]{xcolor}
\usepackage[colorlinks,linkcolor=RawSienna,citecolor=OliveGreen,hypertexnames=true]{hyperref} 
\usepackage{mathtools}
\usepackage{euscript}
\usepackage[all]{xy}
\usepackage{dsfont}
\usepackage{graphicx} 
\usepackage[a4paper,text={165mm,254mm},centering]{geometry} 
\usepackage{comment}  
\usepackage{cite}  
\usepackage{thmtools}
\usepackage{enumitem}
\usepackage{letltxmacro} 
\usepackage{nameref}
\usepackage{cleveref}
\usepackage{stmaryrd}
\usepackage{calc}
\usepackage{interval}
\usepackage{manfnt}
\usepackage{tikz-cd}

\makeatletter %color the labels of equations
\let\reftagform@=\tagform@
\def\tagform@#1{\maketag@@@{(\ignorespaces\textcolor{PineGreen}{#1}\unskip\@@italiccorr)}}
\renewcommand{\eqref}[1]{\textup{\reftagform@{\ref{#1}}}}
\makeatother

\declaretheorem[
name=Theorem,
Refname={Theorem,Theorems},
numberwithin=section]{thm}
\declaretheorem[
name=Proposition,
Refname={Proposition,Propositions},
sibling=thm]{proposition}

\declaretheorem[
name=Claim,
Refname={Claim,Claims},
sibling=thm]{claim}

\declaretheorem[
name=Lemma,
Refname={Lemma,Lemmas},
sibling=thm]{lem}

\declaretheorem[
name=Definition,
Refname={Definition,Definitions},
sibling=thm]{dfn}

\declaretheorem[
name=Remark,
Refname={Remark,Remarks},
sibling=thm]{rem}

\newtheorem{thmx}{Theorem}
\renewcommand{\thethmx}{\Alph{thmx}}

\theoremstyle{plain}
\newlist{thmlist}{enumerate}{1}
\setlist[thmlist]{wide = 0pt, labelwidth = 2em, labelsep*=0em, itemindent = 0pt, leftmargin = \dimexpr\labelwidth + \labelsep\relax, noitemsep,topsep = 1ex, font=\normalfont, label=(\roman*), ref=\thethm.(\roman{thmlisti})}

\addtotheorempostheadhook[thm]{\crefalias{thmlisti}{thm}}

\addtotheorempostheadhook[proposition]{\crefalias{thmlisti}{proposition}}
\addtotheorempostheadhook[rem]{\crefalias{thmlisti}{rem}}
\addtotheorempostheadhook[dfn]{\crefalias{thmlisti}{dfn}}

\addtotheorempostheadhook[lem]{\crefalias{thmlisti}{lem}}
\addtotheorempostheadhook[main]{\crefalias{thmlisti}{main}}

\newlist{thmenum}{enumerate}{1} % also creates a counter called 'propenumi'
\setlist[thmenum]{wide = 0pt, labelwidth = 2em, labelsep*=0em, itemindent = 0pt, leftmargin = \dimexpr\labelwidth + \labelsep\relax, noitemsep,topsep = 1ex, font=\normalfont, label=(\roman*), ref=\thethmx.(\roman{thmenumi})}%{label=\alph*), ref=\thethmx~(\alph*)}
\crefalias{thmenumi}{thmx}

\crefname{lem}{Lemma}{Lemmas}
\crefname{thm}{Theorem}{Theorems}
\crefname{proposition}{Proposition}{Propositions}
\crefname{dfn}{Definition}{Definitions}
\crefname{rem}{Remark}{Remarks}
\crefname{cor}{Corollary}{Corollaries}
\crefname{corx}{Corollary}{Corollaries}
\crefname{problem}{Problem}{Problems}
\crefname{thmx}{Theorem}{Theorems}
\crefname{claim}{Claim}{Claims}

\crefname{main}{Main Theorem}{Main Theorems}

\def\qb{\mathbb{Q}}

\newcommand{\cT}{\mathcal{T}}

\makeatletter
\newcommand*{\rom}[1]{\expandafter\@slowromancap\romannumeral #1@}
\makeatother

\newcommand{\lowerromannumeral}[1]{\romannumeral#1\relax}

\makeatletter     %Replace the Section ...  by the symbol \S ...  in Cref
\newcommand{\crefnames}[3]{%
	\@for\next:=#1\do{%
		\expandafter\crefname\expandafter{\next}{#2}{#3}%
	}%
}
\makeatother

\crefnames{part,chapter,section}{\S}{\S\S}

\makeatletter
\newsavebox{\@brx}
\newcommand{\llangle}[1][]{\savebox{\@brx}{\(\m@th{#1\langle}\)}% 
  \mathopen{\copy\@brx\kern-0.5\wd\@brx\usebox{\@brx}}}
\newcommand{\rrangle}[1][]{\savebox{\@brx}{\(\m@th{#1\rangle}\)}%
  \mathclose{\copy\@brx\kern-0.5\wd\@brx\usebox{\@brx}}}
\makeatother

\let\oldsection\section% Store \section
\renewcommand{\section}{% Update \section
	\renewcommand{\theequation}{\thesection.\arabic{equation}}% Update equation number
	\oldsection}% Regular \section
\let\oldsubsection\subsection% Store \subsection
\renewcommand{\subsection}{% Update \subsection
	\renewcommand{\theequation}{\thesubsection.\arabic{equation}}% Update equation number
	\oldsubsection}% Regular \subsection

\definecolor{plum}{rgb}{0.8,0.2,0.8}

\def\oc{\mathscr{O}}  
  
 \def\fc{\mathcal{F}}

\def\vc{\mathcal{V}}

\def\dl{\lceil D\rceil}
\def\dlt{\lceil \tilde{D}\rceil}

\def\cb{\mathbb{C}}

\def\Im{\operatorname{Im}}

\def\as{{a^\star}} \def\es{e^\star}

\def\cb{\mathbb{C}}

\def\zbb{\mathbb{Z}}

\def\as{\mathscr{A}}
\def\bs{\mathscr{B}}

\def\gs{\mathscr{G}}
\def\fs{\mathscr{F}}
\def\es{\mathscr{E}}

\def\ls{\mathscr{L}}
\def\ds{\mathscr{D}}

\def\xs{\mathscr{X}}

\def\grf{\mathfrak{gr}}

\makeatletter %contibutors debug
\let\@wraptoccontribs\wraptoccontribs
\makeatother

\begin{document}

\title[Hyperbolicity of bases of log CY families]{Hyperbolicity of bases of log Calabi-Yau families} 

\begin{abstract}
	In this paper, we prove that the quasi-projective base of any maximally variational smooth family of Calabi-Yau klt  pairs is both of log general type, and pseudo Kobayashi hyperbolic.  Moreover, such a base is Brody hyperbolic  if the family is effectively parametrized.
\end{abstract}

\author{Ya Deng} 
%\address{\begin{flushleft}
%		Universit\'e de Strasbourg, Institut de Recherche Math\'ematique Avanc\'ee,
%		\newline UMR 7501 du CNRS, 	7 Rue Ren\'e-Descartes,
%		67084 Strasbourg, France
%	\end{flushleft}
%} 
\thanks{The author is supported by  the Knut and Alice Wallenberg Foundation postdoctoral scholarship.}
\address{
	Department of Mathematical Sciences, Chalmers University of Technology \&  University of Gothenburg, Sweden}
\email{\ yade@chalmers.se, dengya.math@gmail.com}

\urladdr{\ https://sites.google.com/site/dengyamath}
	\date{\today}  
\maketitle

\section{Introduction} \label{sec:intro}
The goal of this paper is to prove the hyperbolicity of bases of maximally variational smooth families of log Calabi-Yau pairs.
\begin{thmx}\label{main}
	Let $f^\circ:(X^\circ,D^\circ)\to V$ be a  smooth family of  log     pairs (cf. \cref{def:smooth}) over a quasi-projective manifold $V$. Assume that  each fiber $(X_y,D_y)$ of $f^\circ$ is   Kawamata log terminal (klt for short) and $K_{X_y}+D_y\equiv_{\mathbb{Q}}0$, and the family is of \emph{maximal variation}, \emph{i.e.} the logarithmic Kodaira-Spencer map $T_V\to R^1f^\circ_*\big(T_{X^\circ/V}(-\log D^\circ)\big)$ is generically injective. Then
	\begin{thmenum}
		\item \label{Viehweg}$V$ is of log general type.
		\item \label{pseudo} $V$ is pseudo Kobayashi hyperbolic, \emph{i.e.} there exists a proper subvariety $Z\subsetneq V$ so that  Kobayashi pseudo distance $d_{V}(p,q)>0$ for any two distinct points $(p,q)$ not both contained in $Z$. In particular, any non-constant holomorphic map $\gamma:\cb\to V$ has image $\gamma(\cb)\subset Z$.
		\end{thmenum}
	\end{thmx}
\cref{Viehweg} is often referred to the \emph{Viehweg hyperbolicity} in the literatures. Although \cref{Viehweg} and \cref{pseudo} are conjecturally to be equivalent by the tantalizing Lang's conjecture \cite[Chapter \rom{8}. Conjecture 1.4]{Lan91}, we cannot conclude one from the other directly at the present time. \cref{main} can be seen as some sort of   \emph{Shafarevich hyperbolicity conjecture} for families of log Calabi-Yau pairs. As first formulated  by Viehweg and Kov\'acs,  Shafarevich's conjecture   for higher dimensional   fibers and parametrizing spaces  states that a family of canonically polarized manifolds of maximal variation has as its base a variety of log general type. Shafarevich hyperbolicity conjecture as well as its generalized formulations   drew a lot of attention for a long time, and much progress has been achieved  during the last two decades, 
cf.  \cite{Kov00,VZ01,VZ02,Kov02,VZ03,KK08a,KK08b,KK10,JK11,Sch12,Pat12,TY14,CP15,CP15b,CP16,PS17,BPW17,Sch17,PTW18,Den18,Den18b,TY18,Taj18,WW19}, to quote only a few. 

The proof of \cref{main} is reduced to the  construction of certain   negatively twisted Higgs bundles (which we call Viehweg-Zuo  Higgs bundle in \cref{def:VZ}) over the base $V$ (see \cref{thm:VZ} below),  following the general strategies in \cite{VZ01,VZ02,VZ03,PS17,PTW18}. Indeed, once the Viehweg-Zuo (VZ for short)  Higgs bundle is established, as is well-known to the experts, \cref{Viehweg} follows from  the  celebrated theorem of Campana-P\u{a}un \cite[Theorem 4.1]{CP15b} on the vast generalization of  generic semipositivity result of Miyaoka, 
%the famous generic semipositivity result of Miyaoka, 
and \cref{pseudo} can be deduced from the author's recent work \cite[Theorem 3.8]{Den18} and \cite[Theorem C]{Den18b} on the construction of  \emph{generically non-degenerate}  Finsler metrics over the base (up to a birational model)  with the holomorphic sectional curvatures   bounded above by a negative constant.

%In \cite{VZ01,VZ02}, Viehweg-Zuo initialed a general strategy to study the Viehweg hyperbolicity conjecture, in which they construct   certain negatively twisted Higgs bundles (which we call Viehweg-Zuo Higgs bundle in \cref{def:VZ}) over the base (up to a birational model). Building on the  work by Zuo \cite{Zuo00} on the negativity of the kernel of the Kodaira-Spencer map  of Hodge bundles, they prove that the symmetric differential forms of the base contains a big invertible sheaf (the so-called \emph{Viehweg-Zuo sheaf} in the literatures). In  \cite[Theorem 4.1]{CP15b}, Campana-P\u{a}un developed a vast generalization of
%the famous generic semipositivity result of Miyaoka, and as a direct  consequence they proved the Viehweg hyperbolicity of the base, once its symmetric differential forms contains the Viehweg-Zuo sheaf. 

%In \cite{VZ01,VZ02,VZ03}, Viehweg-Zuo introduced a 

%The proof of \cref{main}, after the work of \cite{VZ01,VZ02,VZ03,KJ11,PS17}

A complex manifold $V$ is said to be \emph{Kobayashi hyperbolic} if the Kobayashi pseudo distance $d_V$ is a metric, and in particular, $V$ is \emph{Brody hyperbolic}: there exists no non-constant holomorphic map $\cb\to V$. The following result is a direct consequence of \cref{main}.
\begin{thmx}\label{cor:Lang}
	For the log Calabi-Yau family $f^\circ:(X^\circ,D^\circ)\to V$ as in \cref{main}, if   $f^\circ$ is further assumed to be \emph{effectively parametrized}, \emph{i.e.} the logarithmic Kodaira-Spencer map 
$$
	T_{V,y}\hookrightarrow  H^1 \big(X_y,T_{X_y}(-\log D_y)\big)
	$$
	is injective for any $y\in V$,
	then
	\begin{thmenum}
		\item \label{all general} every irreducible subvariety of $V$ is of log general type.
		\item \label{Brody} $V$ is Brody hyperbolic.
	\end{thmenum}
	\end{thmx}
According to another conjecture of Lang  \cite[Conjecture 5.6]{Lan86}, \cref{all general} and \cref{Brody} are also expected to be equivalent.   One might expect  that the base $V$  in \cref{cor:Lang} should be moreover \emph{Kobayashi hyperbolic}, which is indeed the case when  fibers of the effectively parametrized family are 
\begin{itemize}[leftmargin=0.7cm]
	\item compact Riemann surfaces of genus $g\geqslant 2$ \cite{Ahl61,Roy74,Wol86};
	\item projective manifolds with ample canonical bundles  \cite{TY14};
		\item Calabi-Yau manifolds (orbifolds) \cite{BPW17,Sch17,TY18}; 
	\item  projective manifolds with big and nef canonical bundles \cite{Den18}.
\end{itemize} 

We conclude the introduction by briefly explaining the proof of \cref{thm:VZ}. Recall that the starting point in the work \cite{VZ02,VZ03,PS17,PTW18,Den18,WW19} is the Viehweg's  $Q_{n,m}$ conjecture on the strong positivity of direct images of pluricanonical bundles, which is known to us when general fibers are of (log) general type or admit  good minimal models, cf.  \cite{Vie83b,Kol87,Kaw85,KP17}. When the fibers are Calabi-Yau klt pairs, Viehweg's  $Q_{n,m}$ conjecture  was only proved recently by Cao-Guenancia-P\u{a}un (see \cref{CGP} below) using very delicate analysis of singular K\"ahler-Einstein metrics. Based on this theorem of  Cao-Guenancia-P\u{a}un, we  apply   Abramovich's $\mathbb{Q}$-mild reduction  for slc families as in \cite{Den18}, to find a good compactification of $V$ so that after replacing the original family by Viehweg's fiber product, one gains enough positivity for the relative dualizing sheaves, which enables us to perform the Viehweg's cyclic cover technique. To construct the desired negatively twisted Higgs bundles,  we mainly follow the general strategies by Viehweg-Zuo \cite{VZ02,VZ03} to generalize their Hodge theoretical methods to the logarithmic setting. However,   different Higgs bundles  defined in the proof of \cref{thm:VZ} are related in a more direct  manner inspired by the recent work of   Popa-Schnell \cite{PS17}  on the alternative construction via   \emph{tautological sections of
	cyclic coverings} (see also the more recent ones   \cite{Wei17,Taj18,WW19}). Our work is also influenced by the work \cite{PTW18}.   Indeed, since the divisor used for the cyclic cover might not be \emph{generically smooth over the base}, we have to perform some ``a priori"   birational modification of the base      so that  certain desingularization of the cyclic cover is smooth over an open set of the base whose complement is a simple normal crossing divisor, by applying an important  technique in \cite[Proposition 4.4]{PTW18}.

In a forthcoming paper we will study the hyperbolicity of bases of   log canonical Calabi-Yau families  using quite different approaches. 
\medskip

\noindent \textbf{Acknowledgements.} I would like to thank Professors Dan Abramovich, Bo Berndtsson and Junyan Cao  for very helpful discussions. I also thank Professors Henri Guenancia, Mihai P\u{a}un and Kang Zuo for their interests on this work.
\section{Positivity of direct image sheaves}
Let us begin this section with the following definitions.
\begin{dfn}
	\begin{thmlist}
		\item \label{def:general type} A quasi-projective irreducible variety $V$ is \emph{of log general type}, if for any desingularization $\tilde{V}\to V$, and any smooth compactification $Y$ of $\tilde{V}$ with the boundary $D:=Y\setminus \tilde{V}$ simple normal crossing,  $K_{Y}+D$ is a big line bundle\footnote{When $V$ is projective, this definition is equivalent to say that the Grauert-Riemenschneider canonical sheaf of $V$ is big.}.
		\item Let $X$ be a quasi-projective variety, and let $D$ be an effective divisor on $X$. $(X,D)$ is called a \emph{smooth log pair} if $X$ is smooth and $D_{\rm red}$ is simple normal crossing.  
		\item  $(X,D)$ is called a \emph{smooth klt pair}, if $(X,D)$ is a smooth log pair  and  all the coefficients of $D$ are in $(0,1)\cap \mathbb{Q}$.
 		\item \label{def:smooth}
	Let $\xs$ and $V$ be   quasi-projective manifolds, and let  $\ds=\sum_{i=1}^{\ell}a_i\ds_i$ be a divisor on $\xs$  with $\sum_{i=1}^{\ell}\ds_i$  simple normal crossing. A surjective projective morphism $(\xs,\ds)\to V$ with connected fibers is called \emph{smooth family of log pairs} if  both $\xs$ and every stratum $\ds_{i_1}\cap \ds_{i_2}\cap\ldots\cap \ds_{i_m}$ are smooth and dominant onto   $V$. Such a $\ds$ is called \emph{relatively normal crossing over $V$}.
	\item Let $(X,D)$ be a smooth log pair,  let $f:(X,D)\to Y$ be a surjective projective morphism and let $V\subset Y$ be a Zariski open set of $Y$. Say $f$ is \emph{smooth over  $V$} if the restriction  $f:(X,D)_{\upharpoonright f^{-1}(V)}\to V$ is  a smooth family of log pairs in the sense of (\lowerromannumeral{4}).
	\end{thmlist}
\end{dfn}

% In what follows, we will always assume that $mK_{X/Y}+mD$ is Cartier, and $mD$ integral, once we say $m$ is \emph{sufficiently large and divisible}.
 
As in \cite{VZ03,PTW18,Den18,WW19}, we have to find a good compactification of $V$ and replace the original family by its fiber product  to construct a cyclic cover, along some divisor in the linear system of   relative pluri-canonical bundles with a sufficiently negative twist. Since the direct images of pluri-canonical bundles  are rank 1 coherent sheaves, the proof can be made much simpler than the general setting, which requires delicate analysis of weak positivity, log canonical threshold and vanishing theorems (or $L^2$-extension theorems). 
\begin{thm}\label{thm:set up}
	For the log Calabi-Yau family $f^\circ:(X^\circ,D^\circ)\to V$ as in \cref{main}, there exists  a   projective compactification $Y\supset V$, a smooth log pair   $(X,D)$ and a surjective morphism $f:X\to Y$ so that
\begin{thmlist}
%\item	\begin{tikzcd}
%	(X^\circ,D^\circ) \arrow[r, hook ] \arrow[d,  "f^\circ"]
%	& (X,D) \arrow[d, "f" ] \\
%	V \arrow[r,hook  ]
%	& Y
%	\end{tikzcd}
%\item $X$ is a projective manifold, and $D$ is an effective divisor with simple normal crossing support.
\item each irreducible component of $D$ is dominant onto $Y$, and  $(X,D)\to Y$ is smooth over $V$;
\item $B:=Y\setminus V$ is simple normal crossing, and $\Delta:=f^*B$ is normal crossing.
\item \label{cyclic divisor} for $m$ sufficiently large and divisible, there exist an ample line bundle $\as$ on $Y$ with $\as-B$ also ample, and an effective divisor
 $
\Gamma\in |mK_{X/Y}+mD-mf^*\as|
$.
	\end{thmlist} 
\end{thm}
%Since $mK_{X/Y}+mD$ is relatively trivial over $V$, $E$ in \cref{cyclic divisor} is supported on $f^*\Delta$.
 
To find the smooth projective compactification $Y$ of $V$ in \cref{thm:set up}, we need  the log version of the \emph{$\mathbb{Q}$-mild reduction} by Abramovich \cite[Corollary A.2]{Den18}. Note that in \cite{VZ02,VZ03} the weakly semi-stable reduction by Abramovich-Karu \cite{AK00} was applied to achieve the same goal. This was substituted in \cite{Den18} by $\mathbb{Q}$-mild reduction, \emph{i.e.}  compactfying the family into Koll\'ar family of slc singularities after taking a finite base change, so that we can control the birational modification of the base precisely. In a recent preprint \cite{WW19}, Wei-Wu applied moduli of stable log pairs  by Kov\'acs-Patakfalvi \cite{KP17} to obtain the similar reduction (which is called  \emph{stable reduction} in \cite[\S 4]{WW19}) for smooth families of klt pairs of log general type. 
\begin{thm}[$\mathbb{Q}$-mild reduction, Abramovich]\label{mild}
	Let $f_0: (S_0,D_0) \to T_0$ be a projective smooth family of log canonical   pairs with $T_0$ quasi-projective manifold. For any given finite subset $Z\subset T_0$,  there exist
	\begin{thmlist}
		\item a compactification $T_0 \subset \underline{\cT}$ with $\underline{\cT}$ a regular projective scheme,  
		\item a simple normal crossings divisor $\Delta \subset \underline{\cT}$ containing $\underline{\cT} \smallsetminus T_0$ and disjoint from $Z$, 
		\item \label{unramified} a finite morphism $W \to \underline{\cT}$ unramified outside $\Delta$, and
		\item \emph{An slc family} $g:(S_W,D_W) \to W$  extending the given family $(S_0,D_0) \times_{\underline{\cT}}W$. That is, each fiber of $g$ is an slc pair, $\omega_{S_W/W}^{[m]}(mD_W)$   is flat over $W$, and commutes with arbitrary base changes for all $m\in \mathbb{Z}_{>0}$ (Koll\'ar condition).
	\end{thmlist}	
	\end{thm} 
We will not recall the basic properties of  $slc$ families, and we refer the readers to \cite{AT16,Pat16} for further details. 
  Let us mention that although  \cref{mild}   is   only stated for the compact setting  in \cite[Corollary A.2]{Den18}, \emph{i.e.} $D_0=\varnothing$, its proof also holds for the log cases, since  \cite[Theorem A.3]{Den18} holds for moduli of Alexeev stable map of slc pairs by \cite[Theorem 1.5]{DR18}.

We will need the following  lemma \cite[Proposition 5.2]{Bou04}, whose proof is a standard   application  of $L^2$-methods (see e.g. \cite{Dem12}).
\begin{lem}\label{lem:Kodaira}
	Let $X$ be a projective manifold equipped with two line bundle $G$ and $L$, and let  $\{x_1,\ldots,x_\ell\}\subset X\setminus \mathbf{B}_+(L)$ be any finite set. If $L$ is big, then  for any  $m\gg 0$, one has the following surjection
	$$
	H^0(X,mL+G)\twoheadrightarrow  \bigoplus_{j=1}^{\ell}J_{x_j}(mL+G)
	$$
	where $J_{x_j}(mL+G):= (mL+G)\otimes \oc_X/\mathfrak{m}_{x_j}$.
	\end{lem}

  Now we   recall a difficult theorem by Cao-Guenancia-P\u{a}un \cite{CGP17} on the proof  of Viehweg's ${Q}_{n,m}$ conjecture for families of  Calabi-Yau  klt  pairs, which is the starting point for the proof of  \cref{thm:set up}. Their proof requires quite delicate analysis on the variation of singular K\"ahler-Einstein metrics. 
  \begin{thm}[\!\!\protect{\cite[Corollary D]{CGP17}}]\label{CGP}
  	For $f^\circ:(X^\circ,D^\circ)\to V$ as in \cref{main}, let $f:(X,D)\to Y$ be any projective compactification, so that $(X,D)$ is a smooth klt   pair, and $Y$ is smooth. Then $ f_*(mK_{X/Y}+mD) ^{\star\star}$ is a big line bundle for $m$ sufficiently large and divisible.
  \end{thm}

Let us begin to prove \cref{thm:set up}. We mainly follow the proof in \cite{VZ03,Den18}, and our proof is also influenced by \cite[Proposition 4.8]{WW19}.
\begin{proof}[Proof of \cref{thm:set up}]
  By the $\qb$-mild reduction in \cref{mild}, we can take a projective compactification $Y\supset V$ with $B:=Y\setminus V$ simple normal crossing, and a finite morphism $\tau:W\to Y$ so that $(X^\circ,D^\circ)\times_YW$  extends to an slc family $(Z,D_Z)\to W$. We denote by $Z^r:=Z\times_W\cdots\times_WZ$ the $r$-folded fiber product of $Z\to W$, and set ${\rm pr}_j:Z^r\to Z$ to be the projection to the $j$-th factor. Write $D_Z^r:=\sum_{j=1}^{r}{\rm pr}_j^*D_Z$. As is well-known, $g^r:(Z^r,D_Z^r)\to W$ is also an slc family. 
\begin{claim}\label{claim}
	For some sufficiently large and divisible $m$, $\omega_{Z/W}^{[m]}(mD_Z)$ is invertible, $g_*(\omega_{Z/W}^{[m]}(mD_Z))$ is a big line bundle, and for any $r\in \mathbb{Z}_{>0}$, one has  
\begin{align}\label{eq:base chagne}
(g^{r})_*\Big(\big(\omega_{Z^{r}/W}^{[m]}(mD_Z^r))=g_*\big(\omega_{Z/W}^{[m]}(mD_Z)\big)^{\otimes r}.
\end{align}
	\end{claim}
\begin{proof}[Proof of \cref{claim}]
By \cite[Proposition 4.4]{AT16}, $(Z,D_Z)$ is a log canonical pair, and in particular $K_Z+D_Z$ is $\mathbb{Q}$-Cartier. Moreover, it was proved    in \cite[Proposition 4.1]{WW19}  that $(Z,D_Z)$ is  even klt.  
 Take a strong log resolution $\mu:\tilde{Z}\to Z$ of $(Z,D_Z)$  with
\begin{align}\label{eq:modification}
\mu^*(K_Z+D_Z)=K_{\tilde{Z}}+\tilde{D}_Z-\tilde{E},
\end{align}
where $\tilde{D}_Z$ and $\tilde{E}$ are both effective $\qb$-divisors, such that $(\tilde{D}_Z+\tilde{E})_{\rm red}$ is simple normal crossing,  $\tilde{E}$ is exceptional and the coefficients of irreducible components of $\tilde{D}_Z$ are all in $(0,1)$. Take $W_0:=\tau^{-1}(V)$, and $Z_0:=g^{-1}(W_0)$. Then $Z_0$ is smooth, and $\mu:\mu^{-1}(Z_0)\to Z_0$ is isomorphism. By the assumption, the logarithmic Kodaira-Spencer map of $\tilde{Z}\to W$ is also generically injective, and by  \cref{CGP} of Cao-Guenancia-P\u{a}un, for $m\gg 0$ large and divisible enough
 $
 \tilde{g}_*(mK_{\tilde{Z}/W}+m\tilde{D}_Z)^{\star\star}
$ is a big line bundle. Here we denote by $\tilde{g}:=g\circ \mu$. On the other hand, by \eqref{eq:modification} one has
$$
\tilde{g}_*(mK_{\tilde{Z}/W}+m\tilde{D}_Z)={g}_*(mK_{{Z}/W}+m {D}_Z),
$$
which is a reflexive sheaf of rank 1, and thus invertible. Therefore, ${g}_*(mK_{{Z}/W}+m {D}_Z)$ is a big line bundle as well. \eqref{eq:base chagne} follows from the base change property of slc families.
\end{proof}
\cref{claim} allows us to replace the original family $(X^\circ,D^\circ)\to V$ by its fiber product to increase the positivity of the direct images. 

Set $B_1\subset Y$ be the proper analytic subset of $Y$ so that $\tau:\tau^{-1}(Y\setminus B_1)\to Y\setminus B_1$ is \'etale. Since $B_2:=\tau\big(\mathbf{B}_+({g}_*(mK_{{Z}/W}+m {D}_Z))\big)$ is a proper subset of $Y$, $V_0:=V\setminus (B_1\cup B_2)$ is a non-empty open set of $V$. Then for any fixed $y\in V_0$, $\tau$ is umramified at $y$, and $\tau^{-1}(y)\subset W\setminus \mathbf{B}_+({g}_*(mK_{{Z}/W}+m {D}_Z))$.  It follows from \cref{lem:Kodaira} that one can take $r\gg 0$ so that
$$
H^0\Big(W,r g_*\big(\omega_{Z/W}^{[m]}(mD_Z)\big)-\tau^*(m\as+mB)\Big) 
$$
generates jet of order 1 at all points in  the finite set $\tau^{-1}(y)$. 
In other words, the locally free sheaf $\tau_*\big(  g_*\big(\omega_{Z/W}^{[m]}(mD_Z)^{\otimes r}\big)\otimes \as^{-m}\otimes \oc_Y(-mB)$ is generated by global sections at $y$, and thus generically over $Y$. Now we replace $(X^\circ,D^\circ)\to V$ by its $r$-folded fiber product, and $(Z,D_Z)\to W$ will also be replaced by the $r$-folded fiber product automatically. We keep the same notation for simplicity. By \cref{claim}, the locally free sheaf 
\begin{align}\label{eq:locally free}
\tau_*g_* \omega_{Z/W}^{[m]}(mD_Z) \otimes \as^{-m}\otimes \oc_Y(-mB)
\end{align}
is \emph{generically generated by global sections}. 

 Take a smooth projective compactification $X\supset X^\circ$, and define $D$ to be the closure of $D^\circ$ in $X$.  After passing to a   blow-up of $X$ with the center in $X\setminus X^\circ $, we can assume that that $f:(X,D)\to Y$ is a surjective morphism,   and  $f^*B+D$ is normal crossing.  In particular, $(X,D)$ is  a smooth  klt pair, which is  smooth over $V$. %By \cite{CGP17}, $f_*(mK_{X/Y}+mD)$ is big for $m\gg 0$. Fix a sufficiently ample line bundle $\as$ and any point $y\in Y\setminus \mathbf{B}_+\big(f_*(mK_{X/Y}+mD)\big)$.  By \cref{unramified} we can make the finite morphism $\tau:W\to Y$ \emph{\'etale} at $y$. 
 Set ${Z}_1:=X\times_YW$, and   $Z_2$ denotes to be its normalization. Since $Z_2$ is birational to ${Z}$, one can even assume that the  log resolution $\tilde{Z}\to Z$ in the proof of \cref{claim} resolves the birational map $Z\dashrightarrow Z_2$, and $\tilde{g}^*(W\setminus W_0)$ is normal crossing.
\begin{equation}\label{dia:inter}
 \begin{tikzcd}
 X \arrow[d, "f"]
 & Z_1\arrow[l, "h"'] \arrow[d, "g_1"']&  \tilde{Z}\arrow[l, "\nu"'] \arrow[ll, bend right, "\phi"'] \arrow[d, "\tilde{g}"]\arrow[r,"\mu"]& Z\arrow[d,"g"] \\
 Y 
 & W  \arrow[l,"\tau"']& W  \arrow[r,"="] \arrow[l,  "="' ]& W
 \end{tikzcd}
 \end{equation}
Since $\tau:W\to Y$ is flat, by \cite[Proof of Lemma 3.3]{Vie83} or \cite[(4.10)]{Mor85}, $Z_1$ is irreducible Gorenstein, $h^*\omega_{X/Y}=\omega_{Z_1/W}$ and $\nu_*\omega_{\tilde{Z}/W}\subset \omega_{Z_1/W}$. By flat base change and the projection formula,  one has 
\begin{align*}
\tilde{g}_*\big(mK_{\tilde{Z}/W}+m\phi^*(D+f^*B)\big)\hookrightarrow(g_1)_*\big(\omega^m_{{Z}_1/W}\otimes h^*\oc_X(mD+mf^*B)\big)\\
=(g_1)_*\Big(h^*\big(\omega^m_{X/Y}\otimes  \oc_X(mD+mf^*B)\big)\Big) 
\stackrel{\simeq}{\to} \tau^*f_*(mK_{X/Y}+mD+mf^*B),
\end{align*}
% By \cite{Vie83}, one has an inclusion
% $$
 %\tilde{g}_*\big(mK_{\tilde{Z}/W}+m\phi^*(D+f^*B)\big)\to %\tau^*f_*(mK_{X/Y}+mD+f^*B),
% $$
 which is an isomorphism over $W_0$. %Take a sufficiently ample divisor  $\as$ on $Y$ so that $\tau_*\oc_W\otimes \as$ is globally generated. By \cref{claim}, there exists $r \gg 0$ so that $r g_*\big(\omega_{Z/W}^m(mD_Z)\big)- \tau^*(2\as+2B)$ is   effective.  Now  we replace the original family $(X^\circ,D^\circ)\to V$ by its $r$-folded fiber product and keep the same notation. By \cref{claim} we have
 %$
%g_*\big(\omega_{Z/W}^m(mD_Z)\big)-r g_*\big(\omega_{Z/W}^m(mD_Z)\big)- \tau^*(2\as+2B)
%$ is effective. 
Since $\mu:\tilde{g}^{-1}(W_0)\to Z_0$ is  an isomorphism, then  $\tilde{g}^{-1}(W_0)\simeq X_0\times_YW$, and thus   
 $
\phi^*(D+f^*B)-\tilde{D}_Z
$ 
is  supported on $\tilde{g}^*(W\setminus W_0)$.  Since $\tilde{D}_Z$ is klt, and $\phi^*f^*B=\tilde{g}^*\tau^*B\geqslant \tilde{g}^*(W\setminus W_0)$,  $
\phi^*(D+f^*B)-\tilde{D}_Z
$ is  thus effective, as also observed in \cite[Proof of Proposition 4.8]{WW19}. One thus has the following morphism 
\begin{align*}
{g}_*(mK_{{Z}/W}+m {D}_Z)=\tilde{g}_*\big(mK_{\tilde{Z}/W}+m\tilde{D}_Z\big)\to \\  \tilde{g}_*\big(mK_{\tilde{Z}/W}+m\phi^*(D+f^*B)\big)\to \tau^*f_*(mK_{X/Y}+mD+mf^*B).
\end{align*}
which is an isomorphic over $W_0$. Hence the morphism
$$
\tau_*{g}_*(mK_{{Z}/W}+m {D}_Z) \to \tau_* \tau^*f_*(mK_{X/Y}+mD+mf^*B)=f_*(mK_{X/Y}+mD+mf^*B)\otimes \tau_*\oc_W.
$$
is isomorphic over $V$.  Thanks to the generically global generation of \eqref{eq:locally free},    $f_*(mK_{X/Y}+mD)  \otimes \as^{-m}\otimes  \tau_*\oc_W$ is also generically generated by global sections. Since the trace map $\tau_*\oc_W\to \oc_Y$ is a splitting surjective morphism, $ f_*(mK_{X/Y}+mD)  \otimes \as^{-m} $ is an effective line bundle. Hence any non-zero section $s\in H^0\big(Y,f_*(mK_{X/Y}+mD)  \otimes \as^{-m} \big)$  will be sent to an effective divisor  $
\Gamma_s\in |mK_{X/Y}+mD-mf^*\as|
$ via the following natural morphism
\begin{align}\label{eq:adjunction}
f^*\big(f_*(mK_{X/Y}+mD)  \otimes \as^{-m} \big)\to mK_{X/Y}+mD-mf^*\as.  
\end{align} 
\end{proof}
\begin{rem}
\begin{thmlist}
		\item \label{non-vanishing}	Since $mK_{X/Y}+mD$ is relatively trivial over $f^{-1}(V)$,    the morphism \eqref{eq:adjunction} is thus an isomorphism over $V$. In particular, the   effective divisor $\Gamma_s\in H^0(X,mK_{X/Y}+mD-mf^*\as)$ induced by $s\in H^0\big(Y,f_*(mK_{X/Y}+mD)  \otimes \as^{-m} \big)$ is supported on  $  f^{-1}\big((s=0)\cup B\big)$.
			\item \label{multiplicity 1}Assume that $\as=2\ls$, where $\ls$ is another very ample line bundle on $Y$. Then for any non-zero $s\in H^0\big(Y,f_*(mK_{X/Y}+mD)  \otimes \as^{-m} \big)$, we can take a general smooth hypersurface $H\in |m\ls|$ which does not contain any prime divisor of $(s=0)$, and $H\cap V\neq \varnothing$. Define $\sigma:=s\cdot s_H\in H^0\big(Y,f_*(mK_{X/Y}+mD)  \otimes \ls^{-m} \big)$, where $s_H\in H^0(Y,m\ls)$ is the canonical divisor defining $H$. By \cref{non-vanishing}, for the   effective divisor $\Gamma_\sigma\in H^0(X,mK_{X/Y}+mD-mf^*\ls)$ induced by $\sigma$, it has at least one irreducible component $P$ with multiplicity one, so that   
			\begin{align*} 
			P\cap f^{-1}(V)=f^*H_{\upharpoonright f^{-1}(V)}.
			\end{align*}  Since $D$ and $ {\Gamma}_\sigma$ do  not have common prime divisors, $P$ is also an irreducible component of $ {\Gamma}_\sigma+  {D}$ with multiplicity one.
	\item \label{rem:morphism} By the proof of \cref{thm:set up} one can show that we have a canonical morphism 
	$$
\Psi:	\tau_*{g}_*(mK_{{Z}/W}+m {D}_Z)\otimes  \as^{-m}\otimes \oc_Y(-mB)\to f_*(mK_{X/Y}+mD)  \otimes \as^{-m}
	$$
	which does not depend on the choice of the intermediate  birational model   $\tilde{Z}$ of $Z$ in \eqref{dia:inter}. Moreover, the linear map
	$$
\Phi:	H^0\big(Z,mK_{{Z}/W}+m {D}_Z-mg^*\tau^*\as-mg^*\tau^*B\big)\to H^0(X,mK_{X/Y}+mD-mf^*\as)
	$$
induced by $\Psi$	is non-trivial.
	\end{thmlist}
\end{rem}

As we mentioned in the end of \cref{sec:intro}, since we  require the discriminant locus of the the new family obtained by the desingularization of certain cyclic cover to be simple normal crossing, we need some sort of ``base change property" of direct images. The following  proposition   follows the ideas in \cite[Proposition 4.4]{PTW18}, and it can be seen as the log version of \cite[Theorem 1.23]{Den18}.\noindent
\begin{proposition}\label{thm:snc}
For $(X,D)\to Y$ and $V\subset Y$ as in \cref{thm:set up}, let $V_0\subset V$ be any non-empty Zariski open set. Assume that  $s\in H^0\big(Y,f_*(mK_{X/Y}+mD)  \otimes \as^{-m}\big)$ is a non-zero  section induced by $s=\Phi(\sigma)$ with $\sigma \in H^0\big(Z,mK_{{Z}/W}+m {D}_Z-mg^*\tau^*\as-mg^*\tau^*B\big)$ and $\Phi$ defined in \cref{rem:morphism}. Then there exists a birational morphism $\mu:\tilde{Y}\to Y$ from a projective manifold $\tilde{Y}$ and a new birational model $\tilde{f}:(\tilde{X},\tilde{D})\to \tilde{Y}$ where $(\tilde{X},\tilde{D})$ is a smooth klt pair satisfying
\[
\begin{tikzcd}
(\tilde{X },\tilde{D})\arrow[d,"\tilde{f}"]\arrow[r] \arrow[rr, bend left, "\phi"]
&  (X\times_Y\tilde{Y})^{\sim}\arrow[d,"\psi"]\arrow[r]  & (X,D)\arrow[d,"f"] \\
\tilde{Y } \arrow[r,   "="]
& \tilde{Y}  \arrow[r,"\mu"] & Y
\end{tikzcd}
\] 
where $ (X\times_Y\tilde{Y})^{\sim}$  is the main component of the  normalization of $X\times_Y\tilde{Y}$,   $\tilde{X}\to (X\times_Y\tilde{Y})^{\sim}$ is some desingularization, and
\begin{enumerate}[leftmargin=0.7cm]
	\item  \label{iso} $\tilde{V}_0:=\mu^{-1}(V_0)\stackrel{\mu}{\to} V_0$ is an isomorphism.
	\item \label{normal} $T_0:=\tilde{Y}\setminus \tilde{V}_0$ and $\tilde{B}:=\mu^*(B)_{\rm red}$ are both simple normal crossing. 
	\item \label{base change} Set $\tilde{X_0}:=(\mu\circ\tilde{f})^{-1}(V)$, and $\tilde{D}_0:=\tilde{D}\cap \tilde{X}_0$. Then $(\tilde{X}_0,\tilde{D}_0)=(X^\circ,D^\circ)\times_Y\tilde{Y}$.  In particular, $\phi:(\tilde{X },\tilde{D})_{\upharpoonright \tilde{f}^{-1}(\tilde{V}_0)}\to (X,D)_{\upharpoonright   f^{-1}(V_0)}$ is also an isomorphism.
	\item \label{normal2}$\tilde{f}^{*}(T_0)+\tilde{D}$ is normal crossing.
	\item \label{coincide} Let $V'\subset V$ be the big open set so that $\mu:\mu^{-1}(V')\to V'$ is an isomorphism. Then  $V'\supset V_0$, and   there exists a section $\tilde{s}\in H^0\big(\tilde{Y},\tilde{f}_*(mK_{\tilde{X}/\tilde{Y}}+m\tilde{D}+m\tilde{E})  \otimes (\mu^*\as)^{-m}\big)$ which coincides with $s$ when restricted to $\tilde{V}':=\mu^{-1}(V')$. Here $\tilde{E}$ is an effective $\tilde{f}$-exceptional divisor with $\tilde{f}(\tilde{E})\cap \tilde{V}'=\varnothing$. % Moreover, $\mu^*s_{\upharpoonright V}=\tilde{s}_{\upharpoonright \mu^{-1}(V)}$.
	\end{enumerate}
	\end{proposition}
\begin{proof}
	Take a birational morphism $\mu:\tilde{Y}\to Y$  so that \eqref{iso} and \eqref{normal} are satisfied. Let  $\tilde{X}\to (X\times_Y\tilde{Y})^{\sim}$ to be a blowing-up  with center in $\psi^{-1}(\tilde{B})$, so that $\phi^*D$, $\tilde{f}^*\tilde{B}$ and $\tilde{f}^*\tilde{B}+\phi^*D$ are all normal crossing.   Set $\tilde{D}$ to be the strict transform of $D$ under $\phi$. Hence $(\tilde{X},\tilde{D})$ is also a  smooth klt pair, which is smooth over $\tilde{V}:=\mu^{-1}(V)$. One can easily see that    \eqref{base change}  is satisfied automatically, and $\tilde{X}_0\cap \tilde{f}^*T_0$  is normal crossing. One can take further blow-up of $\tilde{X}$ with centers in $\tilde{f}^*\tilde{B}$ so that \eqref{normal2} is achieved. 
	
	Take $\tilde{W}$ to be a strong desingularization of $W\times_Y\tilde{Y}$, where $\tau:W\to Y$ is the flat finite morphism introduced in the proof of \cref{thm:set up}. $\tilde{\tau}:\tilde{W}\to \tilde{Y}$ is a generically finite to one surjective morphism, which is finite over $\tilde{V}'$. Hence we can leave out a codimension at least two closed subvariety on $\tilde{Y}\setminus \tilde{V}'$   so that $\tau$ is finite.   Set $(\tilde{Z}, \tilde{D}_{\tilde{Z}}):=(Z,D_Z)\times_W\tilde{W}\stackrel{\tilde{g}}{\to} \tilde{W}$, which is  also  an slc   family. By the base change property, 
	$$
 \delta^*\sigma\in H^0\big(\tilde{Z},mK_{\tilde{Z}/\tilde{W}}+m \tilde{D}_{\tilde{Z}}-m(\tau\circ\delta\circ \tilde{g})^*\as-m(\tau\circ\delta\circ \tilde{g})^*B\big)
	$$
	where $\delta:\tilde{W}\to W$. Since $\mu^*B\geqslant \tilde{B} $, and $\tau\circ\delta=\mu\circ \tilde{\tau}$, then $\delta^*\sigma$ induces a section
	$$
\tilde{\sigma}\in H^0\big(\tilde{Z},mK_{\tilde{Z}/\tilde{W}}+m \tilde{D}_{\tilde{Z}}-m\tilde{g}^*\tilde{\tau}^*(\mu^*\as)-m\tilde{g}^*\tilde{\tau}^*\tilde{B}\big)
	$$
	which is isomorphic to itself when restricted to  $\tilde{\tau}^{-1}(\tilde{V})$. By the fonctorial property stated in \cref{rem:morphism}, $\tilde{\sigma}$ gives rise to a section $$\tilde{s}\in H^0\big(\tilde{Y},\tilde{f}_*(mK_{\tilde{X}/\tilde{Y}}+m\tilde{D})  \otimes \mu^*\as^{-m}\big),$$
	which is isomorphic to $s$ when restricted over $\tilde{V}'\simeq V'$. Note that $\tilde{s}$ only defined over a big open set of $\tilde{Y}$. It extends to a global section, and the $\tilde{f}$-exceptional divisor $\tilde{E}$ might appear.
\end{proof}
\begin{rem}\label{rem:normal}
	\begin{thmlist}
	\item 
When $V_0\subset Y\setminus (s=0)$,  for the non-zero section  $\tilde{s}\in H^0\big(\tilde{Y},\tilde{f}_*(mK_{\tilde{X}/\tilde{Y}}+m\tilde{D}+m\tilde{E})  \otimes \mu^*\as^{-m}\big)$ defined in \eqref{coincide} of \cref{thm:snc}, $(\tilde{s}=0)\cap \mu^{-1}(V_0)=\varnothing$. Let $\tilde{\Gamma}\in |mK_{\tilde{X}/\tilde{Y}}+m\tilde{D}+m\tilde{E}- m(\mu\circ\tilde{f})^*\as|$  be zero divisor defined by $\tilde{s}$. By \cref{non-vanishing}, $\tilde{\Gamma}\subset  \tilde{f}^{-1}(T_0)$. By \eqref{normal2}, $\tilde{f}^{*}(T_0)+\tilde{D}$ is   normal crossing, and thus $\tilde{\Gamma}+\tilde{D}$ is also  normal crossing.
\item \label{irreducible} By  \cref{multiplicity 1}, we can assume there exists an $s\in H^0\big(Y,f_*(mK_{X/Y}+mD)  \otimes \as^{-m}\big)$  so that its zero divisor  $\Gamma_s\in |mK_{X/Y}+mD-mf^*\as|$ contains an irreducible component $P$ with multiplicity one, and 
	\begin{align*} 
P\cap f^{-1}(V)=f^*H_{\upharpoonright f^{-1}(V)}
\end{align*} 
for some smooth hypersurface $H$ on $Y$ with $H\cap V\neq \varnothing$. Since $V'$ in \eqref{coincide} of \cref{thm:snc} is a big open set of $V$, one thus has $H\cap V'\neq \varnothing$.   By \eqref{coincide} in \cref{thm:snc}, there exists  $\tilde{s}\in H^0\big(\tilde{Y},\tilde{f}_*(mK_{\tilde{X}/\tilde{Y}}+m\tilde{D}+m\tilde{E})  \otimes (\mu^*\as)^{-m}\big)$ which coincides with $s$ when restricted to $\tilde{V}':=\mu^{-1}(V')$. Hence for the induced zero divisor $\tilde{\Gamma}_{\tilde{s}}\in |mK_{\tilde{X}/\tilde{Y}}+m\tilde{D}+m\tilde{E}-m\tilde{f}^*\mu^*\as|$, it also contains at least one irreducible component $\tilde{P}$  with multiplicity one, and $\tilde{P}\cap \tilde{f}^{-1}(\tilde{V}')\neq \varnothing$. Such $\tilde{P}$ is indeed the strict transform of $P$ under the birational morphism $\phi:\tilde{X}\to X$.
\end{thmlist}
	\end{rem}

\section{Construction of the Viehweg-Zuo Higgs bundle}
This section is devoted to the construction of certain negatively curved Higgs bundles on the base. This type of  Higgs bundles,   first introduced by Viehweg-Zuo   \cite{VZ01,VZ02,VZ03} and later developed by Popa-Schnell \cite{PS17}, has proven to be a powerful tool in studying the  hyperbolicity of moduli spaces.  Let us give the  definition in an abstract way, following  \cite[Definition 3.1]{Den18}.
\begin{dfn}[Abstract Viehweg-Zuo Higgs bundles]\label{def:VZ}
	Let $V$ be a quasi-projective manifold, and let $Y\supset V$ be a projective compactification of $V$ with the boundary $B:=Y\setminus V$ simple normal crossing. A Viehweg-Zuo Higgs bundle  over $Y$ (or abusively say over $V$) is a logarithmic  Higgs bundle $(\tilde{\es},\tilde{\theta})$ over $Y$  consisting of the following data:
	\begin{thmlist}
		\item  a divisor $T$ on $Y$ so that $B+T$ is simple normal crossing.
		\item \label{VZ big}  A big and nef line bundle $\bs$ over $Y$ with $\mathbf{B}_+(\bs)\subset B\cup T$. 
		\item   A logarithmic Higgs bundle  $({\es}, {\theta}):=\big(\bigoplus_{q=0}^{n}E^{n-q,q},\bigoplus_{q=0}^{n}\theta_{n-q,q}\big)$ induced by the Deligne extension  of a polarized variation of Hodge structure  defined over $Y\setminus (B\cup T)$ with eigenvalues  of residues  lying in $[0,1)\cap{\qb}$.
		\item A sub-Higgs sheaf $(\fs,\eta)\subset (\tilde{\es},\tilde{\theta})$ 
	\end{thmlist}
	satisfying
	\begin{enumerate}[leftmargin=0.7cm]
		\item $(\tilde{\es},\tilde{\theta}):=(\bs^{-1}\otimes {\es},\mathds{1}\otimes {\theta})$. In  particular, $\tilde{\theta}:\tilde{\es}\to \tilde{\es}\otimes \Omega_{Y}\big(\log (B+T)\big)$, and $\tilde{\theta}\wedge\tilde{\theta}=0$.
		\item $(\fs,\eta)$ has log poles only on the boundary $B$, \emph{i.e.} $\eta:\fs\to \fs\otimes\Omega_{Y}(\log B)$.
		\item Write $\tilde{\es}_k:=\bs^{-1}\otimes E^{n-k,k}$, and denote by $\fs_k:=\tilde{\es}_k\cap \fs$. Then the first stage $\fs_0$ of $\fs$ is an \emph{effective line bundle}. In other words, there exists a non-trivial morphism $\oc_Y\to \fs_0$.
	\end{enumerate}
\end{dfn}
Now we are able to state our main result in this section.
\begin{thm}\label{thm:VZ}
	For the $(X^\circ,D^\circ)\to V$ as in \cref{main}, after replacing $V$ by a birational model, there exists a VZ Higgs bundle over some smooth projective compactification $Y$    of $V$.
	\end{thm}
Let us  first show how the above theorem implies \cref{main}.
\begin{proof}[Proof of \cref{main}]
%As we mentioned in \cref{sec:intro}, once the Viehweg-Zuo Higgs bundle is constructed, \cref{Viehweg} immediately follows from the work of Campana-P\u{a}un \cite{CP16}.   Indeed, 
 Once the VZ Higgs bundle is constructed, the proof of \cref{Viehweg} should be well-known to the experts, and we briefly recall the proof for completeness sake.   The first step   is the construction of Viehweg-Zuo (big) sheaf, which is due to Viehweg-Zuo in \cite{VZ02}.  Since $\big(\bigoplus_{q=0}^{n}\fs_q,\bigoplus_{q=0}^{n}\eta_{q}\big)$ is a \emph{sub-Higgs sheaf} of $\big(\tilde{\es},\tilde{\theta}\big)$, as initialed in \cite{VZ01}, for any $q=1,\ldots,n$, the morphism $\eta^q:\fs\to \fs\otimes  {\bigotimes}^q \Omega_Y(\log B)$ defined by iterating $\eta:\fs\to \fs\otimes \Omega_Y(\log B)$ $q$-times induces  
\begin{align}\label{eq:define iterate2}
\fs_0 \to  \fs_q\otimes {\bigotimes}^q \Omega_Y(\log B).
\end{align}
\eqref{eq:define iterate2} 
factors through $\fs_q\otimes S^q\Omega_Y(\log B)$ by $\eta\wedge\eta=0$. Recall that $\fs_0$ is effective, one thus has a morphism 
\begin{align}\label{eq:iterate}
\oc_Y\to \fs_0\to \fs_q\otimes S^q\Omega_Y(\log B).
\end{align} 
We denote by $N_q$ and $K_q$ the kernels  of $ \eta_q:\fs_q\to \fs_{q+1}\otimes \Omega_Y(\log B)$  and $ \theta_{n-q,q}:E_{n-q,q}\to E_{n-q+1,q+1}\otimes \Omega_Y\big(\log (B+T)\big)$ respectively, which are both   torsion free sheaves.    Then $N_q=(\bs^{-1}\otimes K_q)\cap \fs_q$. By the work of  Zuo \cite{Zuo00} (see also \cite{PW16,Bru17,FF17,Bru16} for various generalizations) on the negativity of kernels of Kodaira-Spencer maps of Hodge bundles, $K^*_q$ is \emph{weakly positive} in the sense of Viehweg\footnote{%Based on \cite[Corollary 1.6]{FF17}, after making  some efforts 
By \cite{Bru17}, one can even prove that the Hodge metric   induces a semi-negatively curved singular hermitian metric  for $K_q$ in the sense of Raufi \cite{Rau15} (cf. also \cite{PT14,HPS16}). %See also \Cref{appendix} for a slightly different proof.
}, cf.  \cite[Lemma 4.4.(\lowerromannumeral{5})]{VZ02}. Hence there exists a morphism
$$
\bs\otimes K_q^*\to   N_q^*
$$ 
which is generically surjective. Let $k\in \mathbb{Z}_{\geqslant 0}$ the minimal non-negative integer so that $\eta^k\neq 0$ and $\eta^{k+1}=0$. As proved in \cite[Corollary 4.5]{VZ02}, $k$ must be positive. Indeed, if this is not the case, one has $\oc_Y\subset K_0\otimes \bs^{-1}$, which is not possible. Hence there exists a non-trivial morphism
$$
\oc_Y\to \fs_0\to N_k\otimes S^k\Omega_Y(\log B).
$$
In other words, there exists a non-trivial morphism
\begin{align}\label{eq:big}
\bs\otimes K_k^*\to N_k^*\to  S^k\Omega_Y(\log B).
\end{align}
Since $\bs$ is big and nef, $\bs\otimes K_k^*$ is big in the sense of Viehweg \cite[Definition 1.1.(c)]{VZ02}, and thus for any ample line bundle $\as$ there exists $\alpha\gg 0$ so that 
 $
S^{\alpha}(\bs\otimes K_k^*)\otimes \as^{-1}
$ 
is generated by global sections over a Zariski open set. By \eqref{eq:big} there is a non-zero morphism 
 $$\as\to S^{\alpha k}\Omega_Y(\log B).$$
Such   $\as$ is called the   Viehweg-Zuo  big sheaf in literatures. 

Once the Viehweg-Zuo sheaf $\as$ is constructed, it follows from \cite[Theorem 4.1]{CP15b} that $K_{Y}+B$ is big. \cref{Viehweg} is thus proved. 

 The proof of \cref{pseudo} is exactly the same as \cite[Proof of Theorem B]{Den18b}. In \cite[\S 3]{Den18} we establish an algorithm to  construct a  Finsler metric whose holomorphic sectional curvature is bounded above by a negative constant via VZ Higgs bundles. By proving some sort of \emph{infinitesimal generic Torelli theorem} (cf. \cite[Theorem C]{Den18b}) for VZ Higgs bundles, in \cite[Theorem B]{Den18b} we show that this Finsler metric is  \emph{generically non-degenerate}. The pseudo Kobayashi hyperbolicity of $V$, which is indeed a bimeromorphic property,  follows from the Ahlfors-Schwarz lemma (cf. \cite{Dem97}) immediately. 
\end{proof}
\begin{rem}
 Let us mention that although the Lang conjecture \cite[Chapter \rom{8}. Conjecture 1.4]{Lan91} on the equivalence between pseudo Kobayashi hyperbolicity and being of log general type is quite open at the present time, we know that it holds for   Hilbert modular varieties by \cite{Rou16,CRT17} and subvarities of Abelian varieties \cite{Yam18}. 
	\end{rem}
 We show how \cref{main} implies \cref{cor:Lang}, whose proof is quite standard. 
\begin{proof}[Proof of \cref{cor:Lang}]
We first prove (\lowerromannumeral{1}). Pick any positively dimensional irreducible subvariety $Z\subset V$, and take a   desingularization $\tilde{Z}\to Z$.  Define a base change $f_{\tilde{Z}}:(X^\circ,D^\circ)\times_{ V}\tilde{Z}\to \tilde{Z}$, which is also a smooth family of Calabi-Yau klt pairs. Since the logarithmic Kodaira-Spencer map is fonctorial under the base change, by the effective parametrization assumption,   the logarithmic  Kodaira-Spencer map of $f_{\tilde{Z}}$ is generically injective. Hence $\tilde{Z}$ is of log general type by \cref{Viehweg}, and  (\lowerromannumeral{1}) is proved.

We  will prove (\lowerromannumeral{2}) by contradiction. Suppose that there exists a non-constant holomorphic map $\gamma:\cb\to V$. Set $Z$ to be the Zariski closure of the image $\gamma$, and take a   desingularization $\tilde{Z}\to Z$, which is positively  dimensional, smooth and irreducible. Then $\gamma:\cb\to Z$ lifts to  a Zariski dense entire curve $\tilde{\gamma}:\cb\to \tilde{Z}$. By the above arguments,    $\tilde{Z}$  can be realized as the base of a maximally variational smooth family of Calabi-Yau klt pairs. This is a contradiction, since by \cref{pseudo}, $\tilde{Z}$ cannot have any Zariski dense entire curve.
\end{proof}
Now let us construct the VZ Higgs bundles on some birational modification of the base $V$ in \cref{main}.
\begin{proof}[Proof of \cref{thm:VZ}]
  For the effective divisor $\Gamma\in |mK_{X/Y}+mD-mf^*\as|$   defined in \cref{cyclic divisor}, set $H:=\Gamma+m\dl -mD\in |mK_{X/Y}+m\dl-mf^*\as|$.  By   \cref{multiplicity 1}, we can assume that $H$ contains an irreducible component $P$ with multiplicity one. 
  Let $Z_{\rm cyc}$ be the cyclic cover of $X$ obtained by taking the $m$-th roots along $H:=\Gamma+m\dl -mD\in |mK_{X/Y}+m\dl-mf^*\as|$, and let $Z_{\rm nor}$ be the normalization of $Z_{\rm cyc}$, which is \emph{irreducible} by \cite[Lemma 3.15.(a)]{EV92}. Take  a  desingularization   $Z\to Z_{\rm nor}$, so that the inverse image of $\psi^*H$ is normal crossing,   where $\psi:Z\to X$ is the composition map.  Then $Z$ is \emph{connected}. 
 By \cref{non-vanishing} there exists a  Zariski open set $V_1\subset V$ so that    $\Gamma\cap f^{-1}(V_1)=\varnothing$. Since $\dl$ is relatively normal crossing over $V$, then   $H$ is relatively normal crossing over $V_1$.  Take a smaller Zariski open set $V_0\subset V_1$ so that the composition morphism $g:Z\to Y$ is smooth over $V_0$.

We apply \cref{thm:snc} to find the new birational model $\tilde{f}:(\tilde{X},\tilde{D})\to \tilde{Y}$ with respect to $V_0\subset V$ and a section $\tilde{s}\in H^0\big(\tilde{Y},\tilde{f}_*(mK_{\tilde{X}/\tilde{Y}}+m\tilde{D}+m\tilde{E})  \otimes (\mu^*\as)^{-m}\big)$ satisfying all the properties in \cref{thm:snc}. For zero divisor 
$\tilde{\Gamma}_{\tilde{s}}\in | mK_{\tilde{X}/\tilde{Y}}+m\tilde{D}+m\tilde{E}-m(\mu\circ\tilde{f})^*\as|$ defined by $\tilde{s}$, set $\tilde{H}:=\tilde{\Gamma}_{\tilde{s}}+m\dlt -m\tilde{D}\in |mK_{\tilde{X}/\tilde{Y}}+m\dlt+m\tilde{E}-m(\mu\circ\tilde{f})^*\as|$. By \cref{rem:normal}, $\tilde{H}$ is normal crossing and contains a prime divisor $\tilde{P}$ with multiplicity one. It follows from \eqref{base change} and \eqref{coincide} in \cref{thm:snc} that $\tilde{H}_{\upharpoonright (\mu\circ\tilde{f})^{-1}(V_0)}\simeq H_{\upharpoonright f^{-1}(V_0)}$.   Let $\tilde{Z}_{\rm cyc}$ be the cyclic cover of $\tilde{X}$ obtained by taking the $m$-th roots along $\tilde{H}$, and let $\tilde{Z}_{\rm nor}$ be the normalization of $\tilde{Z}_{\rm cyc}$, which is also irreducible. Then $\tilde{Z}_{\rm nor}\to \tilde{Y}$ and $Z_{\rm nor}\to Y$ is isomorphic when restricted over $\mu^{-1}(V_0)\simeq V_0$. Hence there exists a birational map $Z\dashrightarrow \tilde{Z}_{\rm nor}$, which is regular over $g^{-1}(V_0)$. Take a strong resolution of the indeterminacy $ {\phi} :\tilde{Z}\to \tilde{Z}_{\rm nor}$  of the above birational map
\[\begin{tikzcd}
\tilde{Z} \arrow[r,"\phi"] \arrow[d,"\tilde{\phi}"]\arrow[dd,bend right,"\tilde{g}" ']
& Z \arrow[d] \arrow[dl,dashed]\arrow[dd,bend left,"g"]\\
\tilde{Z}_{\rm nor} \arrow[d,"\eta"]& Z_{\rm nor}\arrow[d]\\
\tilde{Y}\arrow[r,"\mu"] & Y
\end{tikzcd}
\]
so that  $\phi:\phi^{-1}\big(g^{-1}(V_0)\big)\to g^{-1}(V_0)$ is an isomorphism. By the Zariski's  Main theorem,  $\tilde{g}:\tilde{Z}\to \tilde{Y}$ and $g:Z\to Y$ is isomorphic when restricted to $\mu^{-1}(V_0)\simeq V_0$. In particular, $\tilde{g}:\tilde{Z}\to \tilde{Y}$  is smooth over $\mu^{-1}(V_0)$, whose complement is a simple normal crossing divisor in $\tilde{Y}$.
\medskip

To lighten the notation, we replace $(X,D)\to Y$ by the   birational model $(\tilde{X},\tilde{D})\to \tilde{Y}$ and keep the same notation. In summary, it follows from the properties in \cref{thm:snc} that we have the following data: 
\begin{enumerate}[leftmargin=0.7cm,label=(\alph*)]
	\item a surjective morphism $(X,D)\to Y$ which is smooth over $V$, where $(X,D)$ is a smooth klt pair. 
	\item $D+f^*B$ is normal crossing, where $B:=Y\setminus V$.
	\item a normal crossing divisor $H:=\Gamma+m\dl -mD\in |mK_{X/Y}+m\dl-mf^*\as+mE|$ which is relatively   normal crossing over a Zariski open set $V_0\subset V$. Here  $\as$ and $\bs:=\as-B$ are both   big and nef so that $\big(\mathbf{B}_+(\as)\cup \mathbf{B}_+(\bs)\big)\cap V_0=\varnothing$; $E$ is an  $f$-exceptional divisor with $f(E)\cap V_0=\varnothing$. 
	\item There exists a simple normal crossing divisor $T$ so that  $V_0=Y\setminus (B\cup T)$ and $B+T$ is normal crossing. Moreover, $f^*(B+T)+D$ is normal crossing.
	\item For some desingularization $Z$ of the normalization of the cyclic cover of $X$ by taking the $m$-th roots along $H$, $g:Z\to Y$ is smooth over $V_0$, where $g$ is the composition map. Moreover, $Z$ is connected.
	\item  Denote by $\Pi:=g^*(B+T)$ and $H':=\psi^*H$, where $\psi:Z\to X$ is the composition map. Then $\Pi+H'$ is normal crossing.
\end{enumerate}

 Write $D_{\rm red}:=\sum_{i=1}^{r}D_i$, $\Delta:=f^*B$ and $\Sigma:=f^*T$. Leaving out a codimension at least two subscheme, one can assume that 
\begin{enumerate}[leftmargin=0.7cm]
	\item bot $T$ and $B+T$ are smooth;  
	\item both $f:X\to Y$ and $g:Z\to Y$ are flat; in particular, the $f$-exceptional divisor $E$ disappears;
	\item  $\Delta$  (resp. $\Pi$) is  relatively normal crossing over $B$ (resp. $B+T$); 
		\item for any $I=\{i_1,\ldots,i_\ell\}\subset \{1,\ldots,r\}$, the surjective morphism $D_I\to Y$ is also flat, and $D_I\cap \Delta$  is relatively normal crossing over $B$. Here we denote $D_I:=D_{i_1}\cap D_{i_2}\cap \ldots\cap D_{i_\ell}$.
\end{enumerate}
\begin{claim}[good partial compactification]\label{claim:locally free}
	Write $\Omega_{X/Y}\big(\log (\Delta+\dl)\big):= \Omega_X\big(\log (\Delta+\dl)\big)/f^*\Omega_Y(\log B)$, $  \Omega_{Z/Y}(\log \Pi):=\Omega_Z(\log \Pi)/g^*\Omega_Y\big(\log (B+T)\big)$. Then they are all locally free. In other words,  one has the following short exact sequences of locally free sheaves
\begin{align}\label{short1}
0\to f^*\Omega_Y(\log B)\to \Omega_X\big(\log (\Delta+\dl)\big)\to \Omega_{X/Y}\big(\log (\Delta+\dl)\big)\to 0 
\\\label{short4}
0\to g^*\Omega_Y\big(\log (B+T)\big)\to \Omega_Z(\log \Pi)\to \Omega_{Z/Y}(\log \Pi)\to 0.
\end{align} 
	\end{claim} 
\begin{proof}[Proof of \cref{claim:locally free}]
	As is well-know,  the morphism $h:(Z,\Pi)\to (Y,B+T)$ is a  ``good partial compactification" of the smooth morphism $Z\setminus \Pi\to Y\setminus (B\cup T)$ in the sense of \cite[Definition 2.1.(c)]{VZ02}, and one can easily show that   $\Omega_{Z/Y}(\log \Pi)$ is locally free.
	
To prove the local freeness of $\Omega_{X/Y}\big(\log (\Delta+\dl)\big)$, it suffices to show that for any $y_0\in Y$ and $\sigma\in \Omega_Y(\log B)(U)$, where $U\ni y_0$ is some open set with $\sigma(y_0)\neq 0$, $f^*\sigma(x_0)\neq 0$ for any $x_0\in f^{-1}(y_0)$.

Note that there is a unique $I \subset \{1,\ldots,r\}$  so that $x_0\in D_I$ and $x_0\notin D_J$ for any other $J\supsetneq I$.  After reordering, we can assume that $I=\{\ell+1,\ldots,p\}$. Take local coordinates $(x_1,\ldots,x_m)$ around $x$ so that   and  $\Delta_{\rm red}=(x_1\cdots x_\ell=0)$ and $D_i=(x_{i}=0)$ for $i\in I$.  Since $\Delta=f^*B$, the morphism 
$f^*\Omega_Y(\log B)\to \Omega_X\big(\log (\Delta+\dl)\big)$ factors through $\Omega_X(\log \Delta)$. Note that $d\log x_1,\ldots,d\log x_{p},dx_{p+1},\ldots,dx_m$ and $d\log  x_1,\ldots,d\log x_{\ell},dx_{\ell+1},\ldots,dx_m$ form the  local basis for  $\Omega_X\big(\log (\Delta+\dl)\big)$ and $\Omega_X(\log \Delta)$. Hence 
$f^*\sigma=\sum_{j=1}^{\ell} a_j(x)d\log x_j+ \sum_{i=\ell+1}^{m} a_i(x)d  x_i$, where $a_i(x)$ are local holomorphic functions. When we write $f^*\sigma$ in terms of the basis of $\Omega_X\big(\log (\Delta+\dl)\big)$, one has
$$f^*\sigma=\sum_{j=1}^{\ell} a_j(x)d\log x_j+ \sum_{i=\ell+1}^{p} a_i(x)x_id\log  x_i+\sum_{k=p+1}^{m}a_k(x)d x_k.$$
Since $x_0\in D_I=(x_{\ell+1}=x_{\ell+2}=\cdots= x_{p}=0)$, one has
\begin{align}\label{pull-back}
f^*\sigma(x_0)=\sum_{j=1}^{\ell} a_j(x)d\log x_j +\sum_{k=p+1}^{m}a_k(x)d x_k.
\end{align}
When $f^*\sigma$ is seen as the local section in $\Omega_X\big(\log (\Delta+\dl)\big)$, $f^*\sigma(x_0)=0$ if and only if $a_1(x_0)=\cdots=a_{\ell}(x_0)=a_{p+1}(x_0)=\cdots=a_{m}(x_0)=0$. 
On the other hand, since the projective morphism  $f_I:(D_I,D_I\cap \Delta)\to (Y,B)$ satisfies that $f_I$ is flat, $D_I\cap \Delta\to B$ is relatively normal crossing, we thus conclude that $f_I^*\sigma(x_0)\neq 0$.  Since $d\log  x_1,\ldots,d\log x_{\ell},dx_{p+1},\ldots,dx_m$ form the  local basis for  $\Omega_{D_I} (\log D_I\cap \Delta )$,  by \eqref{pull-back}, $f^*\sigma(x_0)=f_I^*\sigma(x_0)\neq 0$, and we conclude that $\Omega_{X/Y}\big(\log (\Delta+\dl)\big)$ is also locally free. 
\end{proof}
By the above claim, such a partial compactification of the smooth family of log pairs $f^{\circ}:(X^\circ,D^\circ)\to V$  can thus be seen as the ``good partial compactification" in the log setting.
\medskip

For any $p\in \mathbb{Z}_{>0}$, let 
\begin{align}\label{eq:Koszul1}
\Omega^{p}_X\big(\log (\Delta+\dl)\big)\otimes \ls^{-1}=\fs^0\supset \fs^1\supset \cdots\supset \fs^p\supset  \fs^{p+1}=0
\end{align}
be the  \emph{Koszul filtration} associated to   \eqref{short1}  twisted with $\ls^{-1}:=-K_{X/Y}-\dl+f^*\as$, defined by 
\begin{align}\label{filtration1}
\fs^i:=\Im \Big(f^*\Omega^i_Y(\log B)\otimes \Omega^{p-i}_X\big(\log (\Delta+\dl)\big)\otimes \ls^{-1}\to \Omega^{p}_X\big(\log (\Delta+\dl)\big) \otimes \ls^{-1}\Big)
\end{align} 
so that the associated graded objects are given by 
\begin{align*}
\grf^i\fs^{\bullet}:=\fs^i/\fs^{i+1}=f^*\Omega^i_Y(\log B)\otimes \Omega^{p-i}_{X/Y}\big(\log (\Delta+\dl)\big)\otimes \ls^{-1}
\end{align*} 
The  \emph{tautological exact sequence} associated to \eqref{eq:Koszul1} is defined by
\begin{equation}\label{eq:tauto1}
\begin{tikzcd}  [column sep=-2em]
&[-1em]  f^*\Omega_Y(\log B)\otimes \Omega^{p-1}_{X/Y}\big(\log (\Delta+\dl)\big)\otimes \ls^{-1} &    & \Omega^{p}_{X/Y}\big(\log (\Delta+\dl)\big)\otimes \ls^{-1}  &\\
0\arrow[r] &  \grf^1\fs^{\bullet}\arrow[u,equal] \arrow[r]   &  \fs^0/\fs^2\arrow[r] &    \grf^0\fs^{\bullet}\arrow[r]   \arrow[u,equal] &  0 
\end{tikzcd}
\end{equation}
By taking higher direct images $\mathbf{R}f_*$ of \eqref{eq:tauto1},
the connecting morphisms of the associated long exact sequences induce $ 
 {F}^{p,q}   \xrightarrow{\tau_{p,q}}  {F}^{p-1,q+1}\otimes \Omega_Y(\log B)   $, 
where    we denote
$$
F^{p,q}:=R^qf_*\Big(\Omega_{X/Y}^p \big(\log (\Delta+\dl)\big)\otimes \ls^{-1} \Big)\Big/{\rm torsion}. $$

Recall that $\psi:Z\to X$ denotes to be the composition map, and $H':=\psi^*H$ is normal crossing. As introduced by Popa-Schnell \cite{PS17} and developed in \cite{WW19}, the tautological section of $\ls$ induces a morphism $
 \psi^*\ls^{-1}\to \oc_Z(-H'_{\rm red})$. 
Since $\psi^*(\Delta+\dl)_{\rm red}\subset (\Pi+H')_{\rm red}$,   for every $i\in\zbb_{>0}$ one has a morphism
 \begin{align*} 
  \psi^* \Omega^i_{X} \big(\log (\Delta+\dl)\big)   \to    \Omega^i_Z \big(\log (\Pi+H' )\big).
\end{align*}  
 Hence there exists a natural morphism
\begin{align}\label{eq:embedd}
\Xi:\psi^*\Big(\Omega^i_{X} \big(\log (\Delta+\dl)\big)\otimes \ls^{-1}\Big)  \to    \Omega^i_Z \big(\log (\Pi+H' )\big)\otimes  \oc_Z(-H'_{\rm red})\subset  \Omega^i_Z (\log \Pi),
\end{align}  
as similarly shown in \cite{PS17,Wei17,Taj18,WW19}.

Pulling back \eqref{short1} by $\psi^*$, we have a short exact sequence of locally free sheaves 
\begin{equation}\label{short5} 
0\to g^*\Omega_Y(\log B)\to \psi^*\Omega_X\big(\log (\Delta +\dl)\big)\to \psi^*\Omega_{X/Y}\big(\log (\Delta +\dl)\big)\to 0  
\end{equation}
In a similar way as \eqref{eq:Koszul1}, we associate  \eqref{short4} and $\eqref{short5}$ with two filtrations 
\begin{align}\label{eq:Koszul2}
&\Omega^{p}_Z(\log \Pi)=\gs^0\supset \gs^1\supset \cdots\supset \gs^p\supset  \gs^{p+1}=0\\\label{eq:Koszul3}
&\psi^*\Big(\Omega^p_{X} \big(\log (\Delta+\dl)\big)\otimes \ls^{-1}\Big) =\tilde{\fs}^0\supset \tilde{\fs}^1\supset \cdots\supset \tilde{\fs}^p\supset  \tilde{\fs}^{p+1}=0
\end{align}
defined by 
\begin{align*}
&\gs^i:=\Im \Big(g^*\Omega^i_Y\big(\log (B+T)\big)\otimes \Omega^{p-i}_Z(\log \Pi)\to \Omega^{p}_Z(\log \Pi) \Big),\\
&\tilde{\fs}^i:=\Im \bigg(g^*\Omega^i_Y(\log B)\otimes \psi^*\Big(\Omega^{p-i}_{X} \big(\log (\Delta +\dl)\big)\otimes \ls^{-1}\Big)\to \psi^*\Big(\Omega^p_{X} \big(\log (\Delta +\dl)\big)\otimes \ls^{-1}\Big) \bigg).
\end{align*}
Their associated graded objects are thus given by
\begin{align}\nonumber &\grf^i\gs^{\bullet}:=\gs^i/\gs^{i+1}=g^*\Omega^i_Y\big(\log (B+T)\big)\otimes \Omega^{p-i}_{Z/Y}(\log \Pi)\\\label{grade}
&\grf^i\tilde{\fs}^{\bullet}:=\tilde{\fs}^i/\tilde{\fs}^{i+1}=g^*\Omega^i_Y(\log B)\otimes \psi^*\Big(\Omega^{p-i}_{X/Y} \big(\log (\Delta +\dl)\big)\otimes \ls^{-1}\Big)=\psi^*\grf^i\fs^{\bullet}.
\end{align} 
One can easily show that  $\Xi$ defined in \eqref{eq:embedd} is compatible with the filtration structures $\Xi:\tilde{\fs}^\bullet\to \gs^\bullet$ in \eqref{eq:Koszul2} and \eqref{eq:Koszul3}.   It thus  induces a morphism   between their graded terms $\grf^i\tilde{\fs}^{\bullet}\to \grf^i\gs^{\bullet}$, %it   induces a morphism from \eqref{short5} to \eqref{short4}, which 
and in particular,  a morphism   between the following short exact sequences
\begin{equation}\label{eq:tauto2}
\begin{tikzcd}  [column sep=tiny]
&[-2em] g^*\Omega_Y(\log B)\otimes \psi^*\Big(\Omega^{p-1}_{X/Y}\big(\log (\Delta+\dl)\big)\otimes \ls^{-1}\Big) &[-4em]        &[-1em]     \psi^*\Big(\Omega^{p}_{X/Y}\big(\log (\Delta+\dl)\big)\otimes \ls^{-1}\Big)  &\\
0\arrow[r] &[-4em]\grf^1\tilde{\fs}^{\bullet}\arrow[r]\arrow[d]\arrow[u,equal]&[-7em] \tilde{\fs}^0/\tilde{\fs}^2\arrow[r]\arrow[d]&[-1em] \grf^0\tilde{\fs}^{\bullet} \arrow[r]\arrow[d]\arrow[u,equal]& 0\\ 
0\arrow[r] &[-6.5em] \grf^1\gs^{\bullet}\arrow[d,equal]\arrow[r] & [-7em]\gs^0/\gs^2\arrow[r] &  [-1em] \grf^0\gs^{\bullet}\arrow[r] \arrow[d,equal]& [-2em]0\\
&[-4em]  g^*\Omega_Y\big(\log (B+T)\big)\otimes \Omega^{p-1}_{Z/Y}(\log \Pi) &[-1em]   &[-4em]    \Omega^{p}_{Z/Y}(\log \Pi)  &
\end{tikzcd}
\end{equation}
Pushing forward \eqref{eq:tauto2} by $\mathbf{R}g_*$,
the edge morphisms   induce
\begin{equation}\label{dia:two Higgs1}
\begin{tikzcd}
\tilde{F}^{p,q}  \arrow[d,"\rho_{p,q}"] \arrow[r,"\varphi_{p,q}"]      &     \tilde{F}^{p-1,q+1}\otimes \Omega_Y(\log B)   \arrow[d,"\rho_{p-1,q+1}\otimes \iota"]\\ 
E_0^{p,q}  \arrow[r,"\theta'_{p,q}"]   &		  E_0^{p-1,q+1}\otimes \Omega_Y\big(\log (B+T)\big) 
\end{tikzcd}
\end{equation}
where  we denote by $E_0^{p,q}:=R^qg_*\big(\Omega_{Z/Y}^p (\log \Pi)\big)$, which is locally free by a theorem of Steenbrink \cite{Ste77} (see  also \cite{Zuc84,Kol86,Kaw02,KMN02} for various generalizations), and $$\tilde{F}^{p,q}:=R^qg_*\bigg(\psi^*\Big(\Omega^p_{X/Y} \big(\log (\Delta +\dl)\big)\otimes \ls^{-1}\Big)\bigg)\Big/{\rm torsion}.$$   $\iota:\Omega_Y(\log B)\hookrightarrow \Omega_Y\big(\log(B+T)\big)$ denotes the natural inclusion. Let us mention that the similar construction as $\big(\bigoplus_{p+q=\ell}\tilde{F}^{p,q},\bigoplus_{p+q=\ell} \varphi_{p,q}\big)$ is made by Taji in his work \cite{Taj18} on a conjecture of Kebekus-Kov\'acs.

Recall that $Z$ is a desingularization of the normalization $Z_{\rm nor}$ of the cyclic cover of $X$ by taking the $m$-th roots along the normal crossing divisor $H$. Write $\psi:Z\xrightarrow{\delta} Z_{\rm nor}\xrightarrow{\phi}X$.  As is well-known, $Z_{\rm nor}$ has  rational singularities  (see e.g. \cite[\S 3]{EV92}), and one thus has $R^q\delta_*\oc_Z=0$ for any $q>0$. By  the projection formula and the degeneration of relative Leray spectral sequences, for any locally free sheaf $\es$ on $X$, one has
\begin{align}\label{eq:higher direct}
R^q\psi_*(\psi^*\es)=\es\otimes R^q\psi_*\oc_Z=\es\otimes R^q\phi_*(\delta_*\oc_Z)=0, \quad \forall q>0
\end{align}
thanks to the finiteness of $\phi$. Applying \eqref{eq:higher direct} to \eqref{grade}, for any $q>0$, we have $R^q\psi_*(\grf^i\tilde{\fs}^\bullet)=0$, and therefore, the exactness  of the tautological short exact sequence of $\tilde{\fs}^{\bullet}$  is preserved under the direct images $\psi_*$ as follows: 
\begin{equation} \label{eq:down exact}
\begin{tikzcd}
0\arrow[r] &\psi_*(\grf^1\tilde{\fs}^{\bullet}) \arrow[r]\arrow[d,equal] & \psi_*(\tilde{\fs}^0/\tilde{\fs}^2) \arrow[r] \arrow[d,equal]& \psi_*(\grf^0\tilde{\fs}^{\bullet})  \arrow[r] \arrow[d,equal]& 0\\
0\arrow[r] & \grf^1\fs^{\bullet}\otimes \psi_*\oc_Z\arrow[r] &\fs^0/\fs^2\otimes \psi_*\oc_Z\arrow[r] & \grf^0\fs^{\bullet}\otimes \psi_*\oc_Z \arrow[r] &0  
\end{tikzcd}
\end{equation}
By    the collapse of relative Leray spectral sequences, one has
\begin{align*}
	R^qg_*(\grf^i\tilde{\fs}^\bullet)&\stackrel{\eqref{grade}}{=}	R^qg_*(\psi^*\grf^i\fs^\bullet)\stackrel{\eqref{eq:higher direct}}{=}	R^qf_*\big(\psi_*(\psi^* \grf^i\fs^\bullet)\big) =R^qf_* (\grf^i\fs^\bullet\otimes \psi_*\oc_Z)  \\
	&=\Omega^i_Y(\log B)\otimes R^qf_*   \Big(\Omega^{p-i}_{X/Y} \big(\log (\Delta+\dl)\big)\otimes \ls^{-1}\otimes \psi_*\oc_Z\Big).
\end{align*}
 %\begin{align}\label{eq:pushforward}
%R^qg_*\psi^*\Big(\Omega^p_{X/\bullet} \big(\log (\Delta+\Sigma+\dl)\big)\otimes \ls^{-1}\Big)=R^qf_* \Big(\Omega^p_{X/\bullet} \big(\log (\Delta+\Sigma+\dl)\big)\otimes \ls^{-1}\otimes \psi_*(\oc_Z)\Big),
%\end{align}
%where $\bullet$ stands for ${\rm Spec}(\cb)$ or $Y$. Moreover,  
Therefore, $\big(\bigoplus_{p+q=\ell}\tilde{F}^{p,q},\bigoplus_{p+q=\ell} \varphi_{p,q}\big)$ can also be defined alternatively by    pushing forward  \eqref{eq:down exact} via $\mathbf{R}f_*$, with $\varphi_{p,q}$ the edge morphisms.

By \cite[Corollary 3.11]{EV92}, the cyclic group $G:=\zbb/m\zbb$ acts on $\psi_*\oc_Z$, and one has the decomposition
$$
\psi_*\oc_Z=\oc_X\oplus\bigoplus_{i=1}^{m-1}(\ls^{(i)})^{-1}, \quad \mbox{where}\quad \ls^{(i)}:=\ls^i\otimes \oc_X(-\lfloor \frac{iH}{m}\rfloor).
$$
In particular, the $G$-invariant part $(\psi_*\oc_Z)^G=\oc_X$. Hence one can easily show that  the cyclic group $G$ acts on \eqref{eq:down exact}, whose $G$-invariant part is
\begin{equation*} 
\begin{tikzcd}
0\arrow[r] &\psi_*(\grf^1\tilde{\fs}^{\bullet})^G\arrow[r]\arrow[d,equal] & \psi_*(\tilde{\fs}^0/\tilde{\fs}^2)^G\arrow[r] \arrow[d,equal]& \psi_*(\grf^0\tilde{\fs}^{\bullet})^G \arrow[r] \arrow[d,equal]& 0\\
0\arrow[r] & \grf^1\fs^{\bullet}\arrow[r] &\fs^0/\fs^2\arrow[r] & \grf^0\fs^{\bullet} \arrow[r] &0  
\end{tikzcd}
\end{equation*} 
Therefore, $\big(\bigoplus_{p+q=\ell}F^{p,q},\bigoplus_{p+q=\ell} \tau_{p,q}\big)$ is a direct factor of  $\big(\bigoplus_{p+q=\ell}\tilde{F}^{p,q},\bigoplus_{p+q=\ell} \varphi_{p,q}\big)$. 
  Combing \eqref{dia:two Higgs1}, we have
\begin{equation}\label{dia:two Higgs2}
\begin{tikzcd}
	E_0^{p,q}  \arrow[r,"\theta'_{p,q}"]   &		  E_0^{p-1,q+1}\otimes \Omega_Y\big(\log (B+T)\big) \\
\tilde{F}^{p,q}  \arrow[u,"\rho_{p,q}"] \arrow[r,"\varphi_{p,q}"]      &     \tilde{F}^{p-1,q+1}\otimes \Omega_Y(\log B) \arrow[u,"\rho_{p-1,q+1}\otimes \iota"'] \\  
F^{p,q}   \arrow[r,"\tau_{p,q}"]  \arrow[u,hook]   &		  F^{p-1,q+1}\otimes \Omega_Y(\log B)   \arrow[u,hook]  
\end{tikzcd}
\end{equation}

Set $n$ to be the relative dimension of $X\to Y$.  Note that
\begin{align}\label{inclusion}
F^{n,0}=f_*(K_{X/Y}-\Delta+\Delta_{\rm red}+\dl+\ls^{-1})=f_*\big( \Delta_{\rm red} +f^*(\as-B)\big)\supset \bs,
\end{align}
where $\bs:=\as-B$ is a big and nef  line bundle.   Define $F^{n-q,q}_0:=\bs^{-1}\otimes F^{n-q,q}$, and 
$$\tau_{n-q,q}': \bs^{-1}\otimes F^{n-q,q}\xrightarrow{\mathds{1}\otimes \tau_{n-q,q}} \bs^{-1}\otimes F^{n-q-1,q+1}\otimes \Omega_Y(\log B)$$
By \eqref{dia:two Higgs2}, one has the following diagram:
\begin{align}\label{dia:two Higgs}
\xymatrixcolsep{4.3pc}\xymatrix{
	\bs^{-1}\otimes	E_0^{n-q,q}  \ar[r]^-{\mathds{1}\otimes \theta'_{n-q,q}}   &		\bs^{-1}\otimes E_0^{n-q-1,q+1}\otimes \Omega_Y\big(\log (B+T)\big) \\
	F_0^{n-q,q}  \ar[u]^{\rho'_{n-q,q}}  \ar[r]^-{  \tau'_{n-q,q}}      &     F_0^{n-q-1,q+1}\otimes \Omega_Y(\log B)   \ar[u]_{\rho'_{n-q-1,q+1}\otimes \iota}}
\end{align}
where $\iota:\Omega_Y(\log B)\to \Omega_Y\big(\log (B+T)\big) $ is the natural inclusion, and $F^{n,0}_0$ is an effective line bundle. Note that all the objects in \eqref{dia:two Higgs} are only defined over a big open set $Y'$ of $Y$.

Write $Z_0:=g^{-1}(V_0)$, which is smooth over $V_0$. The local system
$ R^n g_*\cb_{\upharpoonright Z_0}$  
extends to a locally
free sheaf $\vc$ 
on 
$Y$ (here $Y$ is projective rather than the big open set!) equipped with the logarithmic  connection
$$
\nabla:\vc\to \vc\otimes \Omega_Y\big(\log (B+T)\big),  
$$ 
% where   $\vc$ is  the   quasi-canonical extension of  $ R^n g_*\cb_{\upharpoonright Z_0}$    with 
whose  eigenvalues of the residues   lie  in $[0,1)\cap \mathbb{Q}$ (the so-called \emph{lower canonical extension} in \cite{Kol86}). %By  a   result   of Koll\'ar  \cite[Notation 2.5.(\lowerromannumeral{3})]{Kol86} (based on the \emph{nilpotent orbit theorem}  of Schmid \cite{Sch73} for the unipotent monodromies),  
By \cite{Sch73,CKS86,Kol86}, the Hodge filtration of $ R^n g_*\cb_{\upharpoonright Z_0}$   extends to a filtration $\vc:=\fc^0\supset \fc^1\supset \cdots\supset \fc^{n}$ of \emph{subbundles} so that their graded sheaves $E^{n-q,q}:=\fc^{n-q}/\fc^{n-q+1}$ 
are also  {locally free}, and  there exists
$$
\theta_{n-q,q}:E^{n-q,q}\to E^{n-q-1,q+1}\otimes \Omega_{Y}\big(\log (B+T)\big). 
$$  
As mentioned above, $E_{0}^{n-q,q}$ is locally free by   Steenbrink's theorem. By a theorem of Katz \cite{Kat71}, we know that $(\bigoplus_{q=0}^{n}E_0^{n-q,q},\bigoplus_{q=0}^{n}\theta'_{n-q,q})= (\bigoplus_{q=0}^{n}E^{n-q,q},\bigoplus_{q=0}^{n}\theta_{n-q,q})_{\upharpoonright Y'}$, hence  it can be extended to  the whole projective manifold $Y$ defined by $(\bigoplus_{q=0}^{n}E^{n-q,q},\bigoplus_{q=0}^{n}\theta_{n-q,q})$.   For every $q=0,\ldots, n$, we replace $F_0^{n-q,q}$ by its reflexive hull and thus  the morphisms  
$\tau'_{n-q,q}$,  $\rho'_{n-q,q}$ and the diagram \eqref{dia:two Higgs} extends to the whole $Y$. 

To finish the construction,  we have to introduce the   sub-Higgs sheaf of $\big(\bigoplus_{q=0}^{n}\bs^{-1}\otimes E^{n-q,q},\bigoplus_{q=0}^{n}\mathds{1}\otimes  {\theta}_{n-q,q}\big)$ in \cref{def:VZ}. % following \cite[Corollary 4.5]{VZ02} (cf. also \cite{PTW18}). Write $\tilde{\theta}_{n-q,q}:=\mathds{1}\otimes\theta_{n-q,q}$ for short.
 For each $q=0,\ldots,n$, we define a coherent torsion-free sheaf $\fs_q:=\rho'_{n-q,q}(F_0^{n-q,q})\subset \bs^{-1}\otimes E^{n-q,q}$. 
%By \cite[Lemma 4.4.(\lowerromannumeral{4})]{VZ02}, $\rho_{n,0}$ is an injection, and thus $\tilde{F}^{n,0}\simeq F^{n,0}\supset \oc_Y$. 
 By \eqref{dia:two Higgs}, one has
$$
\mathds{1}\otimes {\theta}_{n-q,q}: \fs_q\to \fs_{q+1}\otimes \Omega_{Y}(\log B),
$$
and  let us denote $\eta_{q}$ by the restriction of  $\mathds{1}\otimes {\theta}_{n-q,q}$ to $\fs_q$.  Then $\big(\bigoplus_{q=0}^{n} \fs_q,\bigoplus_{q=0}^{n}\eta_{q}\big)$ is a sub-Higgs sheaf of $\big(\bigoplus_{q=0}^{n}\bs^{-1}\otimes E^{n-q,q},\bigoplus_{q=0}^{n}\mathds{1}\otimes {\theta}_{n-q,q}\big)$. By \eqref{inclusion}, there exists a morphism $\oc_Y\to \fs_0$ which is an isomorphism over $V_0$. The VZ Higgs bundle is therefore constructed.
	\end{proof}
\begin{rem}
%The intermediate Higgs bundle   $\big(\bigoplus_{q=0}^{n}  F^{n-q,q},\bigoplus_{q=0}^{n}\tau_{n-q,q}\big)$ constructed in the proof of \cref{thm:VZ} is slightly different from that in \cite{VZ02,VZ03}. 
As mentioned above, the  morphism $\Xi$ defined in \eqref{eq:embedd} was first introduced by Popa-Schnell \cite{PS17}, and was later generalized to the log setting in \cite{Wei17,WW19}. This morphism inspires us to construct an intermediate Higgs bundle $\big(\bigoplus_{q=0}^{n}  \tilde{F}^{n-q,q},\bigoplus_{q=0}^{n}\varphi_{n-q,q}\big)$, which  relates   $\big(\bigoplus_{q=0}^{n}  F^{n-q,q},\bigoplus_{q=0}^{n}\tau_{n-q,q}\big)$  with $\big(\bigoplus_{q=0}^{n}  E_0^{n-q,q},\bigoplus_{q=0}^{n}\theta'_{n-q,q}\big)$ in a more   direct manner. 	 In the above proof,    we do not require the divisor $H$ for cyclic cover to be \emph{generically smooth} over the base\footnote{As pointed out by Zuo, when the base is a curve, it has already been studied   in \cite[\S 3]{VZ06}.}, which is more flexible than the original construction in \cite{VZ02,VZ03}. 	One can also see that we reduce the construction of VZ Higgs bundles over a log pair $(Y,B)$ to  the existence of the data (a)-(f).
	\end{rem}

\bibliographystyle{amsalpha}
\bibliography{biblio}

\end{document}